\documentclass[12pt,reqno]{amsart}
\usepackage[margin=2.0cm]{geometry}

\usepackage{amsmath,amssymb,latexsym,amsthm,amstext}
\usepackage{bbm}
\usepackage[nobysame,alphabetic,initials]{amsrefs}
\usepackage{enumerate}

\usepackage{graphicx}
\usepackage{todonotes}

\usepackage{pgfplots}
\pgfplotsset{compat=1.15}
\usepackage{mathrsfs}
\usetikzlibrary{arrows}
\usetikzlibrary{intersections,angles,quotes}
\usepackage{float}

\usepackage{hyperref}

\usepackage{orcidlink}
\usepackage{marginnote} % for notes, delete when final
\usepackage{soul} % for highlighting  texts
\usepackage{cancel} % for crossing out in math mode

\usepackage{mathtools}
\mathtoolsset{showonlyrefs}

\newtheorem{theorem}{Theorem}[section]

\newtheorem{lemma}[theorem]{Lemma}

\newtheorem{corollary}[theorem]{Corollary}

\theoremstyle{remark}
\newtheorem{case}{Case}

\newcommand{\R}{{\mathbb R}}

\newcommand{\hyp}{{\mathbb H}}
\newcommand{\sph}{{\mathbb S}}
\newcommand{\M}{\mathcal{M}}

\DeclareMathOperator{\arcosh}{\mathrm{arcosh}}

\DeclareMathOperator{\area}{\mathrm{area}}

\makeatletter
\DeclareFontFamily{U}{tipa}{}
\DeclareFontShape{U}{tipa}{m}{n}{<->tipa10}{}
\newcommand{\arc@char}{{\usefont{U}{tipa}{m}{n}\symbol{62}}}%

\newcommand{\arc}[1]{\mathpalette\arc@arc{#1}}

\newcommand{\arc@arc}[2]{%
  \sbox0{$\m@th#1#2$}%
  \vbox{
    \hbox{\resizebox{\wd0}{\height}{\arc@char}}
    \nointerlineskip
    \box0
  }%
}
\makeatother

\title[Ball convex P\'al inequality]{P\'al's isominwidth inequality for ball convex bodies in planes of constant curvature}

\author[F. Fodor]{Ferenc Fodor$^{\orcidlink{0000-0001-9747-1981}}$}

\address{Bolyai Institute, University of Szeged, Aradi v\'ertan\'uk tere 1, 6720 Szeged, Hungary}

\email{fodorf@math.u-szeged.hu}

\author[N. Robock]{Nathan Robock}

\address{Department of Mathematics, University of Calgary, 2500 University Drive NW, Calgary, AB T2N 1N4, Canada}

\email{nathan.robock1@ucalgary.ca}

\author[\'A. Sagmeister]{Ádám Sagmeister$^{\orcidlink{0000-0002-6863-6481}}$}

\address{Bolyai Institute, University of Szeged, Aradi v\'ertan\'uk tere 1, 6720 Szeged, Hungary}

\email{sagmeister@server.math.u-szeged.hu, sagmeister.adam@gmail.com}

\subjclass{52A40, 51M09, 51M10, 52A55, 52A10}
\keywords{convex geometry, spherical geometry, hyperbolic geometry, minimal width, thickness, isominwidth inequality, ball bodies}

\dedicatory{Dedicated to Professor K\'aroly Bezdek on the occasion of his 70th birthday.}

\date{\today}

\begin{document}

\begin{abstract}
P\'al's classical isominwidth inequality states that the regular triangle has minimal area among plane convex bodies of minimal width $w$. A similar result is the Blaschke--Lebesgue inequality that states that Reuleaux triangles minimize the area among bodies of constant width $w$ in the plane. In this paper, we connect these two problems by solving the isominwidth problem for $r$-ball convex bodies in the Euclidean, hyperbolic and spherical planes.
\end{abstract}

\maketitle

\section{Introduction}\label{sec:intro}

Many of the classical problems of convex geometry are connected to the concept of width and special classes of convex bodies related to width, i.e. bodies of constant width and reduced convex bodies, are widely studied due to their applicability in the theory of convexity. One of these classical results is due to Blaschke \cite{Bla15} and Lebesgue \cite{Leb14}, who proved that among convex bodies of constant width $w>0$, the Reuleaux triangle has minimal area. A similar result was later obtained by P\'al \cite{Pal}, who proved that among convex bodies in the Euclidean plane the area is minimized by the regular triangle if the minimal width is fixed; this is the so-called isominwidth inequality. In both cases, the key tool of the proof is a lemma due to Blaschke stating that these extremal shapes also minimize the inradius with the corresponding width constraint. Although higher dimensional extensions of these problems are still unsolved even in $\R^3$, there are several generalizations in other directions. It is known that the solution of the Blaschke--Lebesgue problem is still the Reuleaux triangle both on the sphere (K. Bezdek \cite{Bez21}) provided that the minimal width is less than $\frac{\pi}{2}$, and in the hyperbolic plane (B\"or\"oczky and Sagmeister \cite{BoS22}). P\'al's problem was also extended to the sphere for minimal width at most $\frac{\pi}{2}$ by K. Bezdek and Blekherman \cite{Bez00}. If the minimal width is greater than $\frac{\pi}{2}$, then the minimizer for both problems is the polar of a Reuleaux triangle (see Freyer and Sagmeister \cite{FS}). Surprisingly, P\'al's problem does not have a minimal solution in any dimension $n\geq 2$ in the hyperbolic space $\hyp^n$, but for h-convex (horocyclically convex) bodies, the so-called horocyclic Reuleaux triangle minimizes the area in $\hyp^2$ for fixed width (see B\"or\"oczky, Freyer and Sagmeister \cite{BoFS}). 

A convex body is called $r$-ball convex if it is the intersection of balls of radius $r$. Note that one can connect all of these problems by considering $r$-ball convex bodies where $r\geq w$ (we also assume $r<\frac{\pi}{2}$ in $\sph^2$). Indeed, for $r=w$, we obtain the Blaschke--Lebesgue inequality; in $\R^2$, P\'al's problem is the limit case as $r\to\infty$; in $\sph^2$, the spherical version of the P\'al-inequality is the limit case $r\to\frac{\pi}{2}$; and in $\hyp^2$, the horocyclic P\'al inequality is obtained by the limit case $r\to\infty$. Our main result is the $r$-ball convex variant of the isominwidth inequality that has been earlier verified by M. Bezdek \cite{Bez09} for disk polygons in $\R^2$. For fixed $w\leq r$, let $T_{w,r}$ denote the convex set that one obtains by connecting the vertices of a suitable regular triangle by radius $r$ circular arcs such that the minimal width of the resulting figure is $w$. We call $T_{w,r}$ a regular $r$-disk triangle, and it is clearly $r$-ball convex.
\begin{theorem}\label{thm:isominwidth:spindleconvex}
Let $K$ be an $r$-ball convex body in $\M^2$ where $\M^2$ stands for one of $\R^2$, $\hyp^2$ and $\sph^2$. Let $w$ denote the minimal width of $K$, and let $w\leq r$ (in the case $\M^2=\sph^2$ we also assume $r<\frac{\pi}{2}$). Then
$$
\area\left(K\right)\geq\area\left(T_{w,r}\right),
$$
%where $T_{w,r}$ is  a regular $r$-disk triangle with the same minimal width as $K$. 
with equality if and only if $K$ is congruent to $T_{w,r}$.
\end{theorem}

We also prove that, under the same hypotheses as in Theorem~\ref{thm:isominwidth:spindleconvex}, the regular  $r$-disk triangle $T_{w,r}$ minimizes the inradius as well. This fact, which is used in the proof of Theorem~\ref{thm:isominwidth:spindleconvex} is an $r$-ball convex analog of Blaschke's classical result in constant curvature planes, and thus it is sufficiently interesting on its own to be stated explicitly.

\begin{theorem}\label{thm:Blaschke}
Let $K\subset\M^2$ be an $r$-ball convex body of minimal width $w$ where $0<w\leq r$ and if $\M^2=\sph^2$, we also assume $r<\frac{\pi}{2}$. %Let $T_{w,r}\subset\M^2$ denote an $r$-ball convex regular triangle of minimal width $w$. 
Then
\[
\varrho\left(K\right)\geq\varrho\left(T_{w,r}\right),
\]
with equality if and only if $K$ is congruent to $T_{w,r}$.
\end{theorem}

We note that our arguments work in the limit cases $r\to\infty$ if $\M^2=\R^2$ and $\M^2=\hyp^2$ and in the limit case $r\to\frac{\pi}{2}$ if $\M^2=\sph^2$, hence we also provide a simpler proof of the previously known (horoconvex) P\'al inequality in $\M^2$.

The paper is organized as follows. In Section~\ref{sec:preliminaries}, we collect the necessary definitions and technical statements for the arguments. Section~\ref{sec:inradius} contains the proof of Theorem~\ref{thm:Blaschke}. Finally, Theorem~\ref{thm:isominwidth:spindleconvex} is proved in Section~\ref{sec:area}.

\section{Preliminaries}\label{sec:preliminaries}

We use the notation $\M^n$ for an $n$-dimensional space of constant curvature equipped with the geodesic metric $d$. For a positive number $r$ and a point $x\in\M^n$, the closed ball of radius $r$, centered at $x$ is $B(x,r)=\{y\in\M^n\colon d(x,y)\leq r\}$. For two points $x,y\in\M^n$ (where we assume that $x$ and $y$ are not antipodal if $\M^n=\sph^n$), the unique geodesic segment with endpoints $x$ and $y$ is denoted by $[x,y]$. The set $K\subset\M^n$ is \emph{convex} if for all $x,y\in K$, $[x,y]\subset K$ (in the spherical case, we also assume that $K$ is contained in an open hemisphere). $K\subset\M^n$ is a \emph{convex body} if it is convex, compact, and has a non-empty interior. The intersection of an arbitrary family of convex sets is still convex in $\M^n$, so one can define the \emph{convex hull} of the set $X\subset\M^n$ (it is also assumed that $X$ is contained in an open hemisphere if $\M^n=\sph^n$) as the intersection of all convex sets in $\M^n$ containing $X$.  An \emph{$r$-arc} with endpoints $a$ and $b$, denoted $\arc{ab}$, is one of the shorter arcs of radius $r$ from point $a$ to $b$ with the length of $\arc{ab}$ denoted by $\ell(\arc{ab})$.

We use the following technical statement in the proof of Theorem~\ref{thm:isominwidth:spindleconvex}.

\begin{lemma}\label{sublem:minimaldistancebetweencircles}
Let $B(c_1,r)$ and $B(c_2,r)$ be distinct disks such that $d(c_1,c_2) <r$, and let $L$ be the line through $c_1$ and $c_2$.  Let $v=L \cap \partial B(c_1,r)$ such that $v$ is not in $B(c_2,r)$ and choose $f \in \partial B(c_1,r) \cap \partial B(c_2,r)$ to be one of the two points in the intersection. Additionally, let $x \in \arc{fv}$ such that $x$ moves along the arc from $f$ to $v$.  Then the closest point to $x$ on $B(c_2,r)$ is $y=[c_2,x] \cap B(c_2,r)$, and $d(x,y)$ is a strictly increasing function of the length of $\arc{fx}$.
\end{lemma}

\begin{proof}
Let $x_1,x_2 \in \arc{fv}$ be such that $\arc{fx_1}$ is shorter than $\arc{fx_2}$. Then $\angle fc_1x_1 < \angle fc_1x_2$, so $\angle c_2c_1x_1 = \angle c_2c_1f + \angle fc_1x_1 <  \angle c_2c_1f + \angle fc_1x_2 = \angle c_2c_1x_2$, and thus $d(c_2,x_1) < d(c_2,x_2)$. Since for $i \in \{1,2\}$, $[c_2,x_i]$ is orthogonal to $B(c_2,r)$, it is clear that the closest points on $B(c_2,r)$ to $x_1$ and $x_2$ are $y_1=[c_2,x_1] \cap B(c_2,r)$ and $y_2=[c_2,x_2] \cap B(c_2,r)$, respectively.  Then, as $d(c_2,y_1) = d(c_2,y_2) = r$, and $d(c_2,x_1) < d(c_2,x_2)$, it follows that $d(x_1,y_1) < d(x_2,y_2)$.
\end{proof}

\subsection{Width of convex bodies in \boldmath$\M^n$}

Let $K\subset\M^n$ be a convex body and let $H$ be a supporting hyperplane of $K$. We use the notation $w(K,H)$ for the width of $K$ with respect to $H$. The definition depends on $\M^n$, but in any case the width is a continuous function of the direction of $H$, and its maximum among all supporting hyperplanes coincides with the diameter of $K$. The \emph{minimal width} of $K$ (also called \emph{thickness}), which is denoted by $w(K)$, is monotonically increasing with respect to containment, i.e.  if $K_1\subseteq K_2$ then $w(K_1)\leq w(K_2)$.

In spherical and hyperbolic spaces, we use the notion of width introduced by Lassak \cites{La15,Las23}. We note that there are several different concepts of width in hyperbolic space that are all interpretations of the Euclidean width, see the papers of G. Horv\'ath \cite{Hor21} and of B\"or\"oczky, Cs\'epai and Sagmeister \cite{BoCsS}. In $\R^n$, $w(K,H)$ is the distance of $H$ and the unique supporting hyperplane to $K$ that is parallel to $H$. If $\M^n=\hyp^n$, $w(K,H)$ is the distance of $H$ and a farthest ultraparallel (or hyperparallel) supporting hyperplane of $K$. If $\M^n=\sph^n$, the width is defined by lunes. A \emph{lune} is the intersection of two hemispheres with different centers that are not antipodal. The \emph{breadth} of $L$ is the angle of the two hyperplanes bounding $L$, and the \emph{corner} of $L$ is the intersection of these two bounding hyperplanes. We say that $L$ is a \emph{supporting lune} of $K$ if $K\subset L$ and $K$ intersects both of the bounding hyperplanes of $L$ outside of its corners. On the sphere, $w(K,H)$ is the breadth of a lune of minimal breadth such that one of its bounding hyperplanes is the supporting hyperplane $H$.

\subsection{\boldmath$r$-ball convexity}
In this paper, we prove analogs of the classical P\'al inequalities in a setting where the notion of convexity is modified. Let $r>0$ be fixed and let $K\subset\M^n$ be a compact set that is contained in a closed ball of radius $r$. If $\M^n=\sph^n$, we also assume $r<\frac{\pi}{2}$. We say that $K$ is \emph{$r$-ball convex} if it is equal to the intersection of all radius $r$ closed balls in which it is contained. Notice that, unless $K$ is the empty set or a singleton, the interior of an $r$-ball convex $K$ is non-empty. Similarly to the classical convex case, we call a compact set $K$ in $\M^n$ an \emph{$r$-ball convex body} if it is $r$-ball convex and its interior is non-empty. In the concept of $r$-ball convexity, the radius $r$ closed balls play a role that is similar to those of closed half-spaces in classical (linear) convexity. An $r$-ball convex body $K$ has the following important support property: at any boundary point $x$, there exists a radius $r$ closed ball $B$ with $x\in B$ and $K\subset B$. In this case, we say that $B$ supports $K$ at $x$. 
For two points $x,y\in\M^n$ with $d(x,y)\leq 2r$, let $[x,y]_r$ denote the intersection of all radius $r$ closed balls that contain both $x$ and $y$. We call the set $[x,y]_r$ the $r$-segment determined by $x$ and $y$.  
If $K\subset\M^n$ is an $r$-ball convex body, then for any $x,y\in K$, $[x,y]_r\subset K$. This is in analogy with the usual definition of convexity using (linear) segments. We note that if $K$ has the latter property, then it is called \emph{$r$-spindle convex}. In our case, $r$-ball convexity and $r$-spindle convexity are equivalent. However, there are more general settings in which these two notions are not equivalent; see, for example, L\'angi, Nasz\'odi and Talata \cite{LNT13}. 
Although $r$-ball convexity does not include classical convexity directly, in the limit as $r\to\infty$ when $\M^n=\R^n$ or $\hyp^n$, and $r\to\pi/2$ when $\M^n=\sph^n$, it reproduces classical convexity in $\R^n$, horoconvexity in $\hyp^n$ and spherical convexity in $\sph^n$.
We note that $r$-ball convexity is known under other names as well; for more properties, a detailed history, and further references, see the recent survey by Bezdek, L\'angi and Nasz\'odi \cite{BLN25}. 

Since intersections of $r$-ball convex bodies are also $r$-ball convex, one can define the $r$-ball convex hull of a set. Let $X\in\M^n$ be a compact set that is contained in a closed ball of radius $r$. Then the $r$-ball convex hull $[X]_r$ is the intersection of all closed balls of radius $r$ that contain $X$. Clearly $[X]_r$ is an $r$-ball convex body. 

The intersection of a finite number of closed balls of radius $r$ is naturally $r$-ball convex. Such objects are called \emph{$r$-ball polyhedra} and, if $n=2$, then we call them \emph{$r$-disk polygons}. It is not difficult to see that the boundary of an $r$-disk polygon $P$ is the union of a finite number of circular arcs of radius $r$ and, unless $P$ is a disk of radius $r$, it has an equal number of arcs, called \emph{sides}, and non-smooth points, called \emph{vertices}. The boundary structure of higher dimensional ball-polyhedra are more complicated; for more information on this topic see, for example, the paper by Bezdek, L\'angi, Nasz\'odi and Papez \cite{BLNP07}. The $r$-ball convex hull of a finite number of points is called an \emph{$r$-ball polytope} and in the case when $n=2$, $r$-ball polytopes are $r$-disk polygons. However, this is not always the case as an $r$-segment for $n\geq 3$ is the intersection of infinitely many balls that cannot be reduced to a finite number. In this paper, we work in constant curvature planes so we are in the case where $n=2$ and will use $r$-disk polygons.

In the last 10 years, extensive research has been done in this and other more general geometric models with modified concepts of convexity. Intersections of congruent balls appear naturally in several problems such as the Kneser--Poulsen conjecture, Reuleaux polyhedra, and bodies of constant width, among others. Ball-polytopes play important roles in best and random approximation of convex bodies with curvature constraints. For the current state-of-the-art see the survey paper by Bezdek, L\'angi and Nasz\'odi \cite{BLN25} and the references therein.

The geometric figure of central importance in our paper is the regular $r$-disk triangle $T_{w,r}$ of minimal width $w$. The following lemma states that the minimal width of $T_{w,r}$ is attained by its ``height''.

\begin{lemma}
Let $T_{w,r} \subset \M^2$ be a regular $r$-disk triangle such that $0 < w \leq r$, then the minimal width of $T_{w,r}$ is the distance between a vertex and the midpoint of the $r$-arc connecting the other two vertices. 
\end{lemma}

\begin{proof}
%Let $T_{w,r} \subset \M^2$ be a regular $r$-ball triangle such that $0 < w \leq r$, and let $v_1$, $v_2$, and $v_3$ be its vertices with $d$ being the distance between any pair of vertices.  We have three cases, $r < d$, $r>d$ and $r=d$.  If $r=d$ then it is easy to see that the distance be between $v_1$ and the midpoint of $\arc{v_2v_3}$ is going to be equal to $r=d$.  For the remaining two cases, it is easiest to compare to a regular $r$-ball triangle with radius $r=d$ and vertices $v_1$, $v_2$, and $v_3$, call this triangle $T_d$.  If $r>d$, then $\arc{v_2v_3}$ is contained in $T_d$, so the distance between $v_1$ and the midpoint of the $r$-arc $\arc{v_2}{v_3}$ is less than $d$, so $w<d<r$ if $r>d$ and the minimal width is the desired distance.  Lastly, if $r<d$ from the properties of intersections of circles of different radii, it should be clear that, other than the points $v_2$ and $v_3$, $\arc{v_2v_3}\not\subset T_d$, and so the distance between $v_1$ and $\arc{v_2v_3}$ is greater than $d$, so $w\geq d>r$, and thus $r \not\geq w$.   Therefore when $0 < w \leq r$ the minimal width of $T_{w,r}$ is the distance between any vertex and the midpoint of the $r$-arc between the other two vertices.
If $w=r$, then $T_{w,r}$ is a body of constant width, so the statement is clear. Assume $r>w$. Let $v_1,v_2,v_3$ be the vertices of $T_{w,r}$. If $m_1$ is the midpoint of the $r$-arc $\arc{v_2v_3}$ and $c_1$ is the center of the $r$-disk containing the arc $\arc{v_2v_3}$ on its boundary, then $w(T_{w,r},H)=w$ where $H$ is the unique supporting line of $T_{w,r}$ at $m_1$. Note that a width-realizing pair of supporting lines supports $T_{w,r}$ at a vertex and at a point of the opposite $r$-arc. Hence, due to the symmetry of $T_{w,r}$, it is sufficient to show that $w(T_{w,r},H')\geq w$ for any supporting line $H'$ at a point of $\arc{v_2v_3}$. Since $r>w$, $v_1$ is closer to the arc $\arc{v_2v_3}$ than $c_1$. The width $w(T_{w,r},H')$ is realized by a strip (or lune) bounded by a tangent line $H'$ to $B(c_1,r)$ and a chord through $v_1$ such that $c_1$ is never an interior point of this strip (or lune). Hence, the width $w(T_{w,r},H')$ is minimal if and only if the distance of $c_1$ from this chord is maximal. This is equivalent to the orthogonality of the chord and $[c_1,v_1]$. This concludes the proof.
\end{proof}

%Consider the regular $r$-ball triangle $T_{w,d}$ with vertices $v_1$, $v_2$, and $v_3$, where $d$ is the pairwise distance between vertices.  Then it is clear by construction that both the distance between vertices and the distance between any vertex and the midpoint of the $r$-arc between the other two vertices is equal to $d$, so $r=d=w$.  Now assume that $r > d$, and let $T$ be the regular $r$-ball triangle with the same vertices as $T_{w,d}$ and consider the arc $\arc{v_1v_2}_r$ of $T$.  It is clear that since $r > d$, $\arc{v_1v_2}_r \subset T_{w,d}$, and so the distance from $v_3$ to the midpoint of $\arc{v_1v_2}_r$ is less than $d$, the distance between vertices of $T$ so the minimal width of $T$ is less than $d$, and is the desired distance.  Lastly. let $r_1 < d$ and let $T_1$ be the regular $r_1$-ball triangle with the same vertices as $T_{w,d}$ and consider the arc $\arc{v_1v_2}_{r_1}$ of $T_1$.  From the properties of intersections of circles of different radii, it should be clear that other than the points $v_1$ and $v_2$, $\arc{v_1v_2}_{r_1} \not\subset T_{w,d}$, and therefore the distance from $v_3$ to the midpoint of $\arc{v_1v_2}_r$ is greater than $d$, and so the minimal width of $T_1$ is greater than or equal to $d$ and additionally $r$.  So if $0 < w \leq r$, the minimal width of $T_{w,r}$ is the distance between any vertex and the midpoint of the $r$-arc between the other two vertices. 
%For a convex body $K\subset\M^n$ and a supporting hyperplane $H$ to $K$, the width of $K$ with respect to $H$ is the distance of the 

\section{Minimizing the inradius}\label{sec:inradius}

For a convex body $K\subset\M^n$, let
$$
\varrho(K)=\max\{R>0\colon B(x,R)\subseteq K\text{ for some }x\in K\}.
$$
We call $\varrho(K)$ the \emph{inradius} of $K$ and a ball $B(x,\varrho(K))\subseteq K$ an \emph{inscribed ball} (or an \emph{incircle} if $n=2$) of $K$. Inscribed balls of convex bodies are not unique in general; however, they are if $\M^n=\sph^n$, as well as for h-convex bodies in $\hyp^n$ (see B\"or\"oczky, Freyer, Sagmeister \cite{BoFS}) and $r$-ball convex bodies in any space of constant curvature.

\begin{lemma}
Let $K\subset\M^n$ be an $r$-ball convex body for some positive $r$ (we also assume $r<\frac{\pi}{2}$ if $\M^n=\sph^n$). Then $K$ has a unique inscribed ball.
\end{lemma}

\begin{proof}
The existence is clear since $K$ is compact and has a non-empty interior. We prove the uniqueness by contradiction. Let $\varrho$ be the inradius of $K$, and let $B(p,\varrho)$ and $B(s,\varrho)$ be two balls inscribed in $K$ with $p\neq s$. Let $m$ be the midpoint of $[p,s]$, $\ell$ be the line through $p$ and $s$, and $\ell'$ be a half-line orthogonal to $\ell$ emanating from $m$. Let $E$ be the $r$-ball convex hull of $B(p,\varrho)$ and $B(s,\varrho)$. Then $E\subseteq K$. We show that $E$ contains a ball of radius $\varrho'$ with $\varrho'>\varrho$. Let $i$ be the intersection point of $\ell'$ and $\partial E$. We set $d=d(p,s)>0$ and $\varrho'=d(m,i)$, and claim that $B(m,\varrho')\subset E$. Let $c$ be the point on the half-line emanating from $i$ through $m$ with $d(c,i)=r$. Then $[c,m,p]$ is a right triangle with hypotenuse $[c,p]$ where $d(c,p)=r-\varrho$. Then
$$
\varrho'=r-d(c,m)>r-d(c,p)=\varrho
$$
from $d>0$, and clearly $\varrho'<r$, so by rotational symmetry and from $\varrho'<r$, $B(m,\varrho')\subset E$.
\end{proof}

Concerning the contact points of a convex body with its inscribed ball, we quote the following result. Let $\mathrm{relint}\, (X)$ denote the relative interior of a set $X\subset\M^n$.

\begin{lemma}[B\"or\"oczky and Sagmeister\cite{BoS22}]
\label{lemma:closest-points-on-the-boundary}
If $K\subseteq\M^n$ is a convex body and $B\left(p,\varrho\left(K\right)\right)\subset K$, then for $2\leq k\leq n+1$ there are points $t_1,\ldots,t_k\in\partial K\cap\partial B\left(p,\varrho\left(K\right)\right)$ such that $\left[t_1,\ldots,t_k\right]$ is a $\left(k-1\right)$-dimensional simplex and $p\in{\rm{relint}}(\left[t_1,\ldots,t_k\right])$.
\end{lemma}

For a circular disk $B=B(z,\varrho)$ with $\varrho<r$ in $\M^2$, and a point $q\in\M^2$ with $d(q,z)<r-\varrho$, we call the closure of $[B\cup{q}]_r\setminus B$ a \emph{cap} of $B$ with \emph{apex} $q$. We say that $D\subset\M^2$ is an \emph{($r$-ball convex) cap-domain} if it is the union of a circular disk and finitely many pairwise non-overlapping caps. The following lemma shows the existence of a particular cap-domain in $r$-ball convex bodies of $\M^2$.

\begin{lemma}\label{lem:nonoverlapping}
Let $K\subset\M^2$ be an $r$-ball convex body of minimal width $w$ where $0<w\leq r$ and if $\M^2=\sph^2$, then $r<\frac{\pi}{2}$. Let $B=B(p,\varrho)$ be the incircle of $K$ with $\varrho<\frac{w}{2}$. Then, there is an $r$-ball convex cap-domain $B\cup C_1\cup C_2\cup C_3$ contained in $K$ such that $C_1,C_2,C_3$ are non-overlapping with apexes $q_1,q_2,q_3$ with $d(p,q_i)=w-\varrho$.
\end{lemma}

\begin{proof}
As $\varrho<\frac{w}{2}$, there are three points $t_1,t_2,t_3\in\partial K\cap\partial B$ such that $p$ is an interior point in the triangle $\left[t_1, t_2,t_3\right]$ by Lemma~\ref{lemma:closest-points-on-the-boundary}. Let $B_1,B_2,B_3$ be the circular disks of radius $r$ containing $B$ such that the boundaries of $B$ and $B_i$ are tangent in the point $t_i$. Since $K$ is $r$-ball convex, $K\subseteq B_1\cap B_2\cap B_3$. Hence, the minimal width of $T=B_1\cap B_2\cap B_3$ is at least $w$ by the monotonicity of the minimal width with respect to containment. Let $z_1,z_2,z_3$ be the vertices of the $r$-ball convex triangle $T$ opposite to the arcs containing $t_1,t_2,t_3$, respectively. Let $H_1,H_2,H_3$ be the supporting lines of $K$ (and $T$) at $t_1,t_2,t_3$, respectively. Since $w(K,H_i)\geq w$, in the case of $\M^2=\R^2$ or $\M^2=\hyp^2$, a farthest supporting hyperplane to $K$ parallel to $H_i$ intersects $T$ in its cap-region with apex $z_i$. Similarly, if $\M^n=\sph^2$, a supporting hyperplane to $K$ that realizes the width $w(K,H_i)$ intersects $T$ in its cap-region with apex $z_i$. Let $y_i$ be such an intersection point. From the triangle inequality and the fact that the minimal width of $T$ is at least $w$, $d(p,y_i)\geq w-\varrho$, there exist points $q_i\in[p,y_i]$ with $d(p,q_i)=w-\varrho$. Since $T$ is convex and $\varrho<\frac{w}{2}$, $q_i\in C_i$. The caps $C_1,C_2,C_3$ are non-overlapping by definition, and it is easy to see that any $r$-cap of $B$ with apex contained in $C_i$ is a subset of $C_i$, so the caps with apexes $q_1,q_2,q_3$ satisfy the desired properties.
\end{proof}

We can derive the following formula for the inradius of $r$-ball convex regular triangles of width $w$.

\begin{lemma}\label{lem:rho0}
Let $T_{w,r}\subset\M^2$ be an $r$-ball convex regular triangle of minimal width $w$ where $0<w\leq r$ and if $\M^2=\sph^2$, we also assume $r<\frac{\pi}{2}$, and let $\varrho_0$ be the inradius of $T_{w,r}$. Then
\begin{equation}
\varrho_0(r,w)=
\begin{cases}
\frac{r+w-\sqrt{\frac{4r^2-(r-w)^2}{3}}}{2},\text{ if }\M^2=\R^2,\\
\frac{r+w-\arcosh\frac{4\cosh r-\cosh(r-w)}{3}}{2},\text{ if }\M^2=\hyp^2,\\
\frac{r+w-\arccos\frac{4\cos r-\cos(r-w)}{3}}{2},\text{ if }\M^2=\sph^2.
\end{cases}
\end{equation}
\end{lemma}

\begin{proof}
We denote the vertices of $T_{w,r}$ by $v_1,v_2,v_3$, and $p$ denotes its incenter. Let $c_3$ be the center of the circle of radius $r$ whose shorter arc $\arc{v_1v_2}$ is on the boundary of $T_{w,r}$. Let $m_3$ be the midpoint of $\arc{v_1v_2}$; see Figure~\ref{fig1}.

\begin{figure}[h]
\begin{center}
\begin{tikzpicture}[line cap=round,line join=round,>=triangle 45,x=1cm,y=1cm,scale=2.4]
\clip(-1.0660254037844387,-0.75) rectangle (1.3412368303899427,1.2);
\coordinate (p) at (0,0);
\coordinate (v1) at (0,1);
\coordinate (v2) at (-0.8660254037844387,-0.5);
\coordinate (v3) at (0.8660254037844387,-0.5);
\coordinate (c3) at (1.1412368303899427,-0.6588933912347493);
\coordinate (m3) at (-0.6025411342290063,0.3478772860449396);
\draw [shift={(1.1412368303899427,-0.6588933912347493)},line width=0.8pt]  plot[domain=2.173389638068634:3.062598117914354,variable=\t]({1*2.0135413545593766*cos(\t r)+0*2.0135413545593766*sin(\t r)},{0*2.0135413545593766*cos(\t r)+1*2.0135413545593766*sin(\t r)});
\draw [shift={(0,1.3177867824694975)},line width=0.8pt]  plot[domain=2.173389638068634:3.062598117914354,variable=\t]({-0.5*2.0135413545593766*cos(\t r)+-0.8660254037844387*2.0135413545593766*sin(\t r)},{0.8660254037844387*2.0135413545593766*cos(\t r)+-0.5*2.0135413545593766*sin(\t r)});
\draw [shift={(-1.1412368303899434,-0.6588933912347478)},line width=0.8pt]  plot[domain=2.173389638068634:3.062598117914354,variable=\t]({-0.5*2.0135413545593766*cos(\t r)+0.8660254037844384*2.0135413545593766*sin(\t r)},{-0.8660254037844384*2.0135413545593766*cos(\t r)+-0.5*2.0135413545593766*sin(\t r)});
\draw [line width=2pt] (p)-- (v2);
\draw [line width=2pt] (c3)-- (v2);
\draw [line width=0.8pt] (m3)-- (p);
\draw [line width=2pt] (p)-- (v3);
\draw [line width=2pt] (v3)-- (c3);
\node[above] at (v1) {$v_1$};
\node[left] at (v2) {$v_2$};
\node[right] at (v3) {$v_3$};
\node[above] at (p) {$p$};
\node[right] at (c3) {$c_3$};
\node[right] at (m3) {$m_3$};
\draw pic["$\frac{2\pi}{3}$", draw=black, angle radius=8mm, angle eccentricity=0.6] {angle=v2--p--v3};
%\draw (0.07047673765881518,-0.5297205607512919) node[anchor=north west] {$r$};
%\draw (1.0109759309275446,-0.5382705534173713) node[anchor=north west] {$r-w$};
%\draw (0.34692650052871443,-0.1592208785545189) node[anchor=north west] {$w-\varrho_0$};
%\draw (-0.2658229738736396,0.19702881586545506) node[anchor=north west] {$\varrho_0$};
\begin{scriptsize}
\draw [fill=black] (0,0) circle (1pt);
\draw [fill=black] (0,1) circle (1pt);
\draw [fill=black] (-0.8660254037844387,-0.5) circle (1pt);
\draw [fill=black] (0.8660254037844384,-0.5) circle (1pt);
\draw [fill=black] (1.1412368303899427,-0.6588933912347493) circle (1pt);
\draw [fill=black] (-0.6025411342290063,0.3478772860449396) circle (1pt);
\end{scriptsize}
\end{tikzpicture}
\end{center}
\caption{The auxiliary triangle $[c_3,p,v_2]$ to obtain $\varrho_0$}
\label{fig1}
\end{figure}
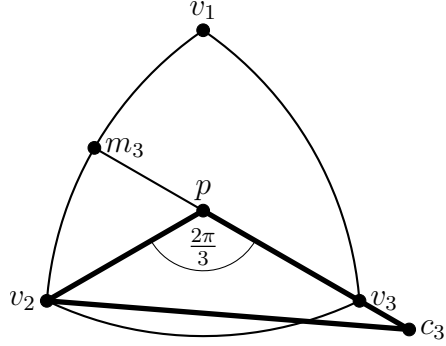

Then $d(c_3,v_2)=d(c_3,m_3)=r$, $d(c_3,p)=d(c_3,m_3)-d(p,m_3)=r-\varrho_0$ and $d(p,v_2)=d(p,v_3)=d(m_3,v_3)-d(m_3,p)=w-\varrho_0$. We also observe that $\angle(c_3,p,v_2)=\angle(v_3,p,v_2)=\frac{2\pi}{3}$. Now we use the law of cosines in the ambient plane.

%\medskip

%\noindent\textbf{Case 1.} \emph{$\M^2=\R^2$.}

\begin{case}
$\M^2=\R^2$.
\end{case}
\begin{equation}
r^2=(w-\varrho_0)^2+(r-\varrho_0)^2+(w-\varrho_0)(r-\varrho_0).
\end{equation}
This can be rearranged as
\begin{equation}
3\varrho_0^2-3(r+w)\varrho_0+(w^2+wr)=0,
\end{equation}
whose solutions for $\varrho_0$ are
\begin{equation}
\frac{r+w\pm\sqrt{\frac{1}{9}\left(9(r+w)^2-12(w^2+wr)\right)}}{2}=\frac{r+w\pm\sqrt{\frac{4r^2-(r-w)^2}{3}}}{2}.
\end{equation}
We note that since $r\geq w$, this leads to two different real solutions, but
\begin{equation}
\frac{r+w+\sqrt{\frac{4r^2-(r-w)^2}{3}}}{2}>w,
\end{equation}
so the only possible solution is
\begin{equation}
\frac{r+w-\sqrt{\frac{4r^2-(r-w)^2}{3}}}{2}.
\end{equation}
It can be easily checked that
\begin{equation}
0<\frac{r+w-\sqrt{\frac{4r^2-(r-w)^2}{3}}}{2}<w.
\end{equation}

\begin{case}
    $\M^2=\hyp^2$.
\end{case}

%\noindent{\bf Case 2 }$\M^2=\hyp^2$

\begin{align*}
\cosh r&=\cosh(w-\varrho_0)\cosh(r-\varrho_0)+\frac{1}{2}\sinh(w-\varrho_0)\sinh(r-\varrho_0)\\
&=\frac{3}{4}\left(\cosh(w-\varrho_0)\cosh(r-\varrho_0)+\sinh(w-\varrho_0)\sinh(r-\varrho_0)\right)\\
&\qquad +\frac{1}{4}\left(\cosh(w-\varrho_0)\cosh(r-\varrho_0)-\sinh(w-\varrho_0)\sinh(r-\varrho_0)\right)\\
&=\frac{3}{4}\cosh(r+w-2\varrho_0)+\frac{1}{4}\cosh(r-w).
\end{align*}
From this we obtain
\begin{equation}
\varrho_0=\frac 12\left (r+w-\arcosh\biggl (\frac{4\cosh r-\cosh(r-w)}{3}\biggr )\right )
\end{equation}
as the unique solution.

%\medskip

%\noindent\textbf{Case 3.} \emph{$\M^2=\sph^2$.}

\begin{case}
    $\M^2=\sph^2$.
\end{case}
\begin{align*}
\cos r&=\cos(w-\varrho_0)\cos(r-\varrho_0)-\frac{1}{2}\sin(w-\varrho_0)\sin(r-\varrho_0)\\
&=\frac{3}{4}\left(\cos(w-\varrho_0)\cos(r-\varrho_0)-\sin(w-\varrho_0)\sin(r-\varrho_0)\right)\\
&\qquad +\frac{1}{4}\left(\cos(w-\varrho_0)\cos(r-\varrho_0)+\sin(w-\varrho_0)\sin(r-\varrho_0)\right)\\
&=\frac{3}{4}\cos(r+w-2\varrho_0)+\frac{1}{4}\cos(r-w).
\end{align*}
From this we can express $\varrho_0$ as
\begin{equation}
\varrho_0=\frac 12\left (r+w-\arccos\biggl (\frac{4\cos r-\cos(r-w)}{3}\biggr)\right ).
\end{equation}
\end{proof}

\begin{corollary}\label{cor:rho0:monotone}
The inradius of $T_{w,r}$ increases in $w$ for any fixed $r$, and decreases in $r$ for any fixed $w$.
\end{corollary}

\begin{proof}
Consider the partial derivatives of $\varrho_0$.
 %Fix $r$, and consider the function $\varrho_0(r,\cdot)\colon(0,r]\to\R$ defined as
%\[
%\varrho_0(w)=
%\begin{cases}
%\frac{r+w-\sqrt{\frac{4r^2-(r-w)^2}{3}}}{2}\text{ if }\M^2=\R^2,\\
%\frac{r+w-\arcosh\frac{4\cosh r-\cosh(r-w)}{3}}{2}\text{ if }\M^2=\hyp^2,\\
%\frac{r+w-\arccos\frac{4\cos r-\cos(r-w)}{3}}{2}\text{ if }\M^2=\sph^2.
%\end{cases}
%\]
%Then
\[
2\frac{\varrho_0(r,w)}{\partial w}=
\begin{cases}
1-\frac{r-w}{\sqrt{9r^2+6rw-3w^2}},\text{ if }\M^2=\R^2,\\
1-\frac{\sinh(r-w)}{\sqrt{(4\cosh r-\cosh(r-w))^2-9}},\text{ if }\M^2=\hyp^2,\\
1-\frac{\sin(r-w)}{\sqrt{9-(4\cos r-\cos(r-w))^2}},\text{ if }\M^2=\sph^2,
\end{cases}
\]
and
\[
2\frac{\varrho_0(r,w)}{\partial r}=
\begin{cases}
1-\frac{3r+w}{\sqrt{9r^2+6rw-3w^2}},\text{ if }\M^2=\R^2,\\
1-\frac{4\sinh r-\sinh(r-w)}{\sqrt{(4\cosh r-\cosh(r-w))^2-9}},\text{ if }\M^2=\hyp^2,\\
1-\frac{4\sin r-\sin(r-w)}{\sqrt{9-(4\cos r-\cos(r-w))^2}},\text{ if }\M^2=\sph^2.
\end{cases}
\]
Note that the numerators in $\frac{\partial\varrho_0(r,w)}{\partial w}$ and $\frac{\partial\varrho_0(r,w)}{\partial r}$ are both non-negative in all three cases. It remains to show that the right-hand sides are always positive.

%\medskip

%\noindent {\textbf{Case 1.} $\M^2=\R^2$.}

\setcounter{case}{0}
\begin{case}
  $\M^2=\R^2$.  
\end{case}
If $r>0$ is fixed, then the monotonicity of $\varrho_0$ with respect to $w$ is equivalent to
\[
(r-w)^2<9r^2+6rw-3w^2\iff 0<8r^2+8rw-4w^2,
\]
which clearly holds as $w\leq r$.

On the other hand, for a fixed $w>0$, the monotonicity of $\varrho_0$ with respect to $r$ is equivalent to
\[
9r^2+6rw-3w^2<(3r+w)^2,
\]
and that is also obvious.
%\medskip

%\noindent{\textbf{Case 2.} $\M^2=\hyp^2$.}

\begin{case}
  $\M^2=\hyp^2$.  
\end{case}
The monotonicity of $\varrho_0$ with respect to $w$ is equivalent to
\[
\sinh^2(r-w)<(4\cosh r-\cosh(r-w))^2-9\iff 0<16\cosh^2 r-8\cosh r\cosh(r-w)-8,
\]
which again follows from $w\leq r$.

The monotonicity of $\varrho_0$ with respect to $r$ is equivalent to
\begin{align*}
(4\cosh r-\cosh(r-w))^2-9&<(4\sinh r-\sinh(r-w))^2\\\iff  1&<\cosh r\cosh(r-w)-\sinh r\sinh(r-w)=\cosh w,
\end{align*}
and this holds trivially.
%\medskip

%\noindent{\textbf{Case 3.} $\M^2=\sph^2$.}

\begin{case}
  $\M^2=\sph^2$.  
\end{case}
Similarly, on the sphere, the monotonicity of $\varrho_0$ with respect to $w$ is equivalent to
\[
\sin^2(r-w)<9-(4\cos r-\cos(r-w))^2\iff 0<8+8\cos r\cos(r-w)-16\cos^2 r,
\]
which is again a consequence of $w\leq r$.

Finally, the monotonicity of $\varrho_0$ with respect to $r$ is equivalent to
\begin{align*}
9-(4\cos r-\cos(r-w))^2&<(4\sin r-\sin(r-w))^2\\\iff 1&>\cos r\cos(r-w)+\sin r\sin(r-w)=\cos w.
\end{align*}
%which is again clear.
\end{proof}

Now we have all the necessary information to show that $r$-ball convex regular triangles minimize the inradius among $r$-ball convex bodies with fixed minimal width.

%\begin{theorem}\label{thm:Blaschke}
%Let $K\subset\M^2$ be an $r$-ball convex body of minimal width $w$ where $0<w\leq r$ and if $\M^2=\sph^2$, we also assume $r\leq\frac{\pi}{2}$. Let $T_{w,r}\subset\M^2$ denote an $r$-ball convex regular triangle of minimal width $w$. Then
%\[
%\varrho\left(K\right)\geq\varrho\left(T_{w,r}\right),
%\]
%with equality if and only if $K$ and $T_{w,r}$ are congruent.
%\end{theorem}

\begin{proof}[Proof of Theorem~\ref{thm:Blaschke}]

Let $\varrho=\varrho(K)$ and $\varrho_0=\varrho(T_{w,r})$. We observe that $\varrho\leq\frac{w}{2}$. By Lemma \ref{lem:rho0}, it can easily be checked that $\varrho_0<\frac{w}{2}$, so we can assume without loss of generality that $\varrho<\frac{w}{2}$. If $\varrho>\varrho_0$, then we have nothing to prove, so we assume otherwise. Let $B\subseteq K$ be the incircle of $K$ centered at $p$. Then, by Lemma \ref{lem:nonoverlapping}, there is a cap-domain $B\cup C_1\cup C_2\cup C_3$ contained in $K$ such that the apexes $q_i$ of the caps $C_i$ are of distance $w-\varrho$ from $p$, and the caps $C_1,C_2,C_3$ are pairwise non-overlapping, and hence a cap $C_i$ covers at most one third of the boundary of $B$.

Consider an $r$-ball convex regular triangle $T'$ of minimal width $w'$ whose incircle is $B$ and vertices are $v_1,v_2,v_3$. Since $T'$ is the union of $B$ and three congruent $r$-caps such that these caps cover the boundary, we obtain $w'-\varrho=d(p,v_i)\geq d(p,q_i)=w-\varrho$. This yields $w'\geq w$, and, in turn, $\varrho_0\leq\varrho$ by Corollary~\ref{cor:rho0:monotone}.

In the case of equality, the points $q_1,q_2,q_3$ are the vertices of an $r$-ball convex regular triangle $T\subseteq K$ of minimal width $w$. On the other hand, since $K$ is $r$-ball convex, $K$ is contained by the three circular disks of radius $r$ that contain $B$ and are tangent to it at $t_1,t_2,t_3$, respectively, where $t_i$ is the contact point of $\partial B$ and $\partial K$ opposite to the cap-domain $C_i$. But the intersection of these three circles is $T$, so in this case $K=T$.
\end{proof}

\section{Minimizing the area}\label{sec:area}

In this section, we prove Theorem~\ref{thm:isominwidth:spindleconvex}.

\begin{lemma}\label{lem:areaofhexagonsincreases}
Let $B\subset\M^2$ be a circular disk centered at the incenter $p$ of $T_{w,r}$ and of radius $\varrho\in\left(\varrho_0,\frac{w}{2}\right)$ where $\varrho_0$ denotes the inradius of $T_{w,r}$. Let $q_1,q_2,q_3$ be points of $\M^2$ such that $d\left(p,q_i\right)=w-\varrho$ and $\angle\left(q_i,p,q_j\right)=\frac{2\pi}{3}$ for $i\neq j$. Let the points $t_i$ be the intersection of the half-line emanating from $q_i$ through $p$ and $\partial B$. Finally, let $Q_{w,r,\varrho}$ be the $r$-ball convex hull of $q_1,q_2,q_3,t_1,t_2,t_3$. Then, the area of $Q_{w,r,\varrho}$ is greater than the area of $T_{w,r}$.
\end{lemma}

\begin{proof}
As $T_{w,r}=Q_{w,r,\varrho_0}$ we can construct $T_{w,r}$ in a convenient way centered at the point $p$.  Let $v_1$ be the vertex of $T_{w,r}$ that is collinear with the points $p$ and $q_1$ and let $m_2 \in \partial B_0 \cap \partial T_{w,r}$ be the point that is collinear with $p$ and $t_2$, where $B_0$ is the incircle of $T_{w,r}$. Let $\alpha = d(m_2,t_2)$ and notice that $d(v_1,q_1) = d(p,v_1)-d(p,q_1) = (w-\varrho_0) - (w-\varrho) = \varrho - \varrho_0 = d(m_2,t_2)$. Let $f$ be the intersection point of $\arc{v_1m_2}$ and $\arc{q_1t_2}$. Due to the symmetry and shared area of $T_{w,r}$ and $Q_{w,r,\varrho}$, we only need to compare the area of the domain bounded by $\arc{q_1f}$, $\arc{fv_1}$ and $[v_1,q_1]$ and the area of the domain bounded by $\arc{t_2f}$, $\arc{fm_2}$ and $[m_2,t_2]$; see Figure~\ref{fig:Twr-and-Qwrrho}.

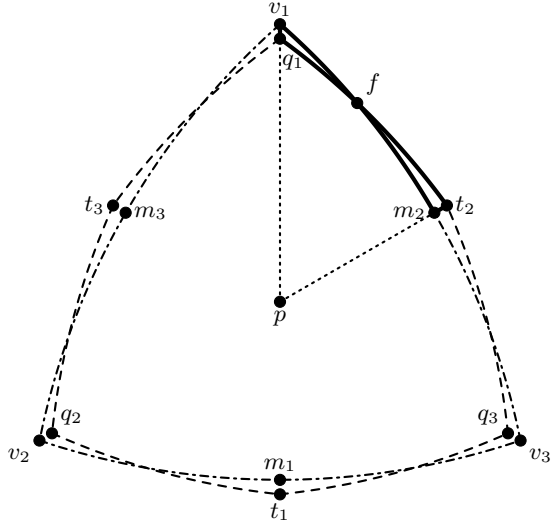
\begin{figure}[h]
\centering
\begin{tikzpicture}[line cap=round,line join=round,>=triangle 45,x=1cm,y=1cm,scale=3.5]
\clip(-1.1,-0.9) rectangle (1.1,1.2);
\coordinate (p) at (0,0);
\coordinate (v1) at (0,1.0550685001861666);
\coordinate (v2) at (-0.9137161238939671,-0.5275342500930831);
\coordinate (v3) at (0.9137161238939671,-0.5275342500930831);
\coordinate (q1) at (0.05,0.97);
\coordinate (q2) at (-0.8660254037844387,-0.5);
\coordinate (q3) at (0.8660254037844387,-0.5);
\coordinate (m1) at (0,-0.6769823073827106);
\coordinate (m2) at (0.5862838761060329,0.33849115369135513);
\coordinate (m3) at (-0.5862838761060329,0.33849115369135513);
\coordinate (t1) at (0,-0.7320508075688772);
\coordinate (t2) at (0.6339745962155613,0.36602540378443843);
\coordinate (t3) at (-0.6339745962155613,0.36602540378443843);
\coordinate (f) at (0.29328693802568606,0.7554289505400491);
\clip(-1.723715513261319,-1) rectangle (1.2796698843114531,1.1384428813092446);
\draw [shift={(1.8974127181395855,-1.0954717435817094)},line width=0.8pt,dashdotted]  plot[domain=2.2937435756730413:2.942244180309947,variable=\t]({1*2.8679257945461294*cos(\t r)+0*2.8679257945461294*sin(\t r)},{0*2.8679257945461294*cos(\t r)+1*2.8679257945461294*sin(\t r)});
\draw [shift={(0,2.190943487163417)},line width=0.8pt,dashdotted]  plot[domain=4.388138678066237:5.036639282703142,variable=\t]({1*2.8679257945461276*cos(\t r)+0*2.8679257945461276*sin(\t r)},{0*2.8679257945461276*cos(\t r)+1*2.8679257945461276*sin(\t r)});
\draw [shift={(0.3001370256235798,2.120126628859801)},line width=0.8pt,dashed]  plot[domain=4.29363517904471:4.60754400612637,variable=\t]({1*2.8679257945461276*cos(\t r)+0*2.8679257945461276*sin(\t r)},{0*2.8679257945461276*cos(\t r)+1*2.8679257945461276*sin(\t r)});
\draw [shift={(-0.3001370256235795,2.120126628859801)},line width=0.8pt,dashed]  plot[domain=4.817233954643008:5.1311427817246695,variable=\t]({1*2.8679257945461276*cos(\t r)+0*2.8679257945461276*sin(\t r)},{0*2.8679257945461276*cos(\t r)+1*2.8679257945461276*sin(\t r)});
\draw [shift={(1.986152032644239,-0.8001370256235804)},line width=0.8pt,dashed]  plot[domain=2.722838852249813:3.036747679331474,variable=\t]({1*2.8679257945461267*cos(\t r)+0*2.8679257945461267*sin(\t r)},{0*2.8679257945461267*cos(\t r)+1*2.8679257945461267*sin(\t r)});
\draw [shift={(1.6860150070206592,-1.3199896032362222)},line width=0.8pt,dashed]  plot[domain=2.199240076651514:2.5131489037331747,variable=\t]({1*2.8679257945461276*cos(\t r)+0*2.8679257945461276*sin(\t r)},{0*2.8679257945461276*cos(\t r)+1*2.8679257945461276*sin(\t r)});
\draw [shift={(-1.98615203264424,-0.8001370256235789)},line width=0.8pt,dashed]  plot[domain=0.1048449742583187:0.41670518836069265,variable=\t]({1*2.867925794546127*cos(\t r)+0*2.867925794546127*sin(\t r)},{0*2.867925794546127*cos(\t r)+1*2.867925794546127*sin(\t r)});
\draw [shift={(-1.6860150070206603,-1.319989603236221)},line width=1.6pt]  plot[domain=0.6284437498566177:0.809098591459293,variable=\t]({1*2.867925794546127*cos(\t r)+0*2.867925794546127*sin(\t r)},{0*2.867925794546127*cos(\t r)+1*2.867925794546127*sin(\t r)});
\draw [shift={(-1.6860150070206603,-1.319989603236221)},line width=1.6pt]  plot[domain=0.809098591459293:0.9423525769382786,variable=\t]({1*2.867925794546127*cos(\t r)+0*2.867925794546127*sin(\t r)},{0*2.867925794546127*cos(\t r)+1*2.867925794546127*sin(\t r)});
\draw [shift={(-1.8974127181395817,-1.095471743581706)},line width=0.8pt,dashdotted]  plot[domain=0.1993484732798451:0.5235987755982985,variable=\t]({1*2.867925794546124*cos(\t r)+0*2.867925794546124*sin(\t r)},{0*2.867925794546124*cos(\t r)+1*2.867925794546124*sin(\t r)});
\draw [shift={(-1.8974127181395817,-1.095471743581706)},line width=1.6pt]  plot[domain=0.5235987755982985:0.7015200780642291,variable=\t]({1*2.867925794546124*cos(\t r)+0*2.867925794546124*sin(\t r)},{0*2.867925794546124*cos(\t r)+1*2.867925794546124*sin(\t r)});
\draw [shift={(-1.8974127181395817,-1.095471743581706)},line width=1.6pt]  plot[domain=0.7015200780642291:0.847849077916752,variable=\t]({1*2.867925794546124*cos(\t r)+0*2.867925794546124*sin(\t r)},{0*2.867925794546124*cos(\t r)+1*2.867925794546124*sin(\t r)});
\draw [line width=0.8pt,dotted] (0,0)-- (0,1);
\draw [line width=1.6pt] (0,1)-- (0,1.0550685001861666);
\draw [line width=0.8pt,dotted] (0,0)-- (0.5862838761060329,0.33849115369135513);
\draw [line width=1.6pt] (0.5862838761060329,0.33849115369135513)-- (0.6339745962155613,0.36602540378443843);
\begin{scriptsize}
\draw [fill=black] (0,0) circle (0.6pt);
\draw [fill=black] (0,1) circle (0.6pt);
\draw [fill=black] (-0.8660254037844387,-0.5) circle (0.6pt);
\draw [fill=black] (0.8660254037844384,-0.5) circle (0.6pt);
\draw [fill=black] (0,1.0550685001861666) circle (0.6pt);
\draw [fill=black] (0,-0.7320508075688772) circle (0.6pt);
\draw [fill=black] (0,-0.6769823073827106) circle (0.6pt);
\draw [fill=black] (-0.9137161238939671,-0.5275342500930831) circle (0.6pt);
\draw [fill=black] (0.9137161238939666,-0.5275342500930836) circle (0.6pt);
\draw [fill=black] (0.5862838761060329,0.33849115369135513) circle (0.6pt);
\draw [fill=black] (-0.5862838761060327,0.3384911536913556) circle (0.6pt);
\draw [fill=black] (0.6339745962155613,0.36602540378443843) circle (0.6pt);
\draw [fill=black] (-0.633974596215561,0.3660254037844389) circle (0.6pt);
\draw [fill=black] (0.29328693802568606,0.7554289505400491) circle (0.6pt);
\node[below] at (p) {$p$};
\node[above] at (v1) {$v_1$};
\node[below left] at (v2) {$v_2$};
\node[below right] at (v3) {$v_3$};
\node[above] at (m1) {$m_1$};
\node[left] at (m2) {$m_2$};
\node[right] at (m3) {$m_3$};
\node[below] at (q1) {$q_1$};
\node[above right] at (q2) {$q_2$};
\node[above left] at (q3) {$q_3$};
\node[below] at (t1) {$t_1$};
\node[right] at (t2) {$t_2$};
\node[left] at (t3) {$t_3$};
\node[above right] at (f) {$f$};
\end{scriptsize}
\end{tikzpicture}
\caption{The $r$-disk triangle $T_{w,r}$ and the $r$-disk hexagon $Q_{w,r,\varrho}$}
\label{fig:Twr-and-Qwrrho}
\end{figure}

%In order to make this comparison, we use the following sublemma.

%\begin{sublemma}\label{sublem:minimaldistancebetweencircles}
%Let $L$ be the line through the centres of $B(s_1,r)$ and $B(s_2,r)$, and let $v=L \cap \partial B(s_1,r)$ and $u=L \cap \partial B(s_2,r)$ such that $v$ is not in $B(s_2,r)$ and $u$ is not in $B(s_1,r)$, and, without loss of generality, choose $f \in \partial B(s_1,r) \cap \partial B(s_2,r)$ to be one of the two points in the intersection. Additionally, let $x \in \arc{fv}$ such that $x$ moves along the arc from $f$ to $v$,  then the closest point on $B(s_2,r)$ is the point $y \in B(s_2,r)$ such that $y=[s_2,x] \cap B(s_2,r)$ and $d(x,y)$ is a strictly increasing function of the length of $\arc{fx}$.
%\end{sublemma}

%\begin{proof}
%Let $x_1,x_2 \in \arc{fv}$ such that $\arc{fx_1}$ is shorter than $\arc{fx_2}$, then $\angle fs_1x_1 < \angle fs_1x_2$, which gives $\angle s_2s_1x_1 = \angle s_2s_1f + \angle fs_1x_1 <  \angle s_2s_1f + \angle fs_1x_2 = \angle s_2s_1x_2$ which implies that $d(s_2,x_1) < d(s_2,x_2)$, finally, since $[s_2,x_i]$ is orthogonal to $B(s_2,r)$, it is clear that the closest points on $B(s_2,r)$ to $x_1$ and $x_2$ are $y_1=[s_2,x_1] \cap B(s_2,r)$ and $y_2=[s_2,x_2] \cap B(s_2,r)$.  Then, as $d(s_2,y_1) = d(s_2,y_2) = r$, and since $d(s_2,x_1) < d(s_2,x_2)$, $d(x_1,y_1) < d(x_2,y_2)$.
%\end{proof}

Since $[p,m_2]$ is orthogonal to $\arc{v_1m_2}$, $\alpha$ is the shortest distance between $\arc{fm_2}$ and $\arc{ft_2}$ at $t_2$.  If $[p,q_1]$ is orthogonal to $\arc{fq_1}$, the shortest distance between $\arc{fq_1}$ and $\arc{fv_1}$ at $v_1$ is also $\alpha$ and thus the area of the two sets are equal as, by Lemma \ref{sublem:minimaldistancebetweencircles},  $\ell(\arc{fq_1})=\ell(\arc{fm_2})$ and $\ell(\arc{fv_1})=\ell(\arc{ft_2})$.  However, if $[p,q_1]$ is orthogonal to $\arc{fq_1}$, then there exists a point $c_1$ collinear with $[p,q_1]$ such that $c_1$ is the center of the circle that contains $\arc{q_1t_2}$, so $d(c_1,q_1)=r$, and $d(c_1,p)=r-w+\varrho$.  Then $d(c_1,p)+d(p,t_2)=r-w+\varrho+\varrho < r=d(c_1,t_2)$, as $t_2$ is also a point on $B(c_1,r)$, but this violates the triangle inequality, and thus, $[p,q_1]$ is not orthogonal to $\arc{fq_1}$.
%for $\varrho \in (\varrho_0,\frac{w}{2})$ the area of $Q_{w,r,\varrho}$ is not equal to that of $T_{w,r}$.

Since $[p,q_1]$ is not orthogonal to $\arc{fq_1}$, then $\alpha$ is greater than the shortest distance between  $\arc{fv_1}$ and $\arc{fq_1}$ at the point $v_1$, and therefore, by the monotonicity of the function presented in Lemma~\ref{sublem:minimaldistancebetweencircles}, $\ell(\arc{fv_1}) < \ell(\arc{ft_2})$.  It remains to show that $\ell(\arc{fq_1}) < \ell(\arc{fm_2})$.  Assume otherwise, if $\ell(\arc{fq_1}) \geq \ell(\arc{fm_2})$, then there exists a point $q_1^- \in \arc{fq_1}$ such that $\ell(\arc{fq_1^-})= \ell(\arc{fm_2})$.  Likewise, there exists a point $v_1^+$ such that $\arc{fv_1^+}$ contains $\arc{fv_1}$ such that $\ell(\arc{fv_1^+}) = \ell(\arc{ft_2})$, and therefore $d(v_1^+,q_1^-)=\alpha$ is the shortest distance between $\arc{fq_1^-}$ and $\arc{fv_1^+}$ at $v_1^+$.  

Let $c_2$ be the center of the circle that contains $\arc{fv_1^+}$, and let $L$ be the line that passes through $c_2$ and $q_1$. Since $\alpha$ is not the shortest distance at $v_1$, we can find the point $b=\arc{fv_1^+}\cap L$. By construction, $d(b,q_1)$ is the shortest distance between $\arc{fq_1}$ and $\arc{fv_1^+}$ at $b$.  It is easy to see due to the lengths of the respective arcs that $[b,q_1]$ intersects $[v_1^+,q_1^-]$.  However, this leads to a contradiction, as both $d(b,q_1)$ and $d(v_1^+,q_1^-)$ are the shortest distance between the arcs $\arc{fq_1}$ and $\arc{fv_1^+}$ at $b$ and $v_1^+$, respectively.%  As $[c_2,v_1^+]$ contains $[v_1^+,q_1^-]$ and $[c_2,b]$ contains $[b,q_1]$ so $[c_2,b]$ and $[c_2,v_1^+]$ intersect twice, which is impossible.
Therefore, $\ell(\arc{fq_1}) < \ell(\arc{fm_2})$, and due to the reflective symmetry of two intersecting circles with identical radius through the line that passes through the two points of intersection of the boundaries of said circles, we can find points $v_1' \in \arc{ft_2}$ and $q_1' \in \arc{fm_2}$, such that $\ell(\arc{fv_1'}) = \ell(\arc{fv_1})$ and $\ell(\arc{fq_1'}) = \ell(\arc{fq_1})$ (see Figure~\ref{fig:reflection}). Then, the region bounded by $\arc{v_1'f}$, $\arc{fq_1'}$ and $[q_1',v_1']$ is a proper subset of the region bounded by $\arc{t_2f}$, $\arc{fm_2}$ and $[m_2,t_2]$%It should be easy to see that the area bounded by $\arc{fx}$, $\arc{fy}$, and $[x,y]$ is contained in the area bounded by $\arc{fm_2}$, $\arc{ft_2}$, and $[m_2,t_2]$
, and thus the area of $T_{w,r} \setminus Q_{w,r,\varrho}$ is less than that of $Q_{w,r,\varrho} \setminus T_{w,r}$, and therefore the area of $Q_{w,r,\varrho}$ is greater than that of $T_{w,r}$.

\begin{figure}[H]
\centering
\begin{tikzpicture}[line cap=round,line join=round,>=triangle 45,x=1cm,y=1cm,scale=2]
\clip(-1.21,-1.1) rectangle (1.21,1.2);
\coordinate (f) at (0.1,0.9792237906882774);
\coordinate (c2) at (-0.2012686136287686,0);
\coordinate (c12) at (0.2012686136287686,0);
\coordinate (v1) at (-0.5537262719083826,0.9358277614604361);
\coordinate (q1) at (-0.3004268259698176,0.863281154818844);
\coordinate (v1') at (0.5537262719083826,0.9358277614604361);
\coordinate (q1') at (0.3004268259698176,0.863281154818844);
\coordinate (m2) at (0.5411562225186984,0.6699293714052301);
\coordinate (t2) at (0.7367730079530657,0.8464447891576709);
\draw [line width=0.8pt, dashed] (0.001513949107973231,-1.1)--(0.001513949107973231,1.2);
\draw [line width=2pt] (0.5411562225186984,0.6699293714052301)-- (0.7367730079530657,0.8464447891576709);
\draw [line width=0.8pt,dotted] (-0.2012686136287686,0)-- (0.5411562225186984,0.6699293714052301);
\draw [line width=1.6pt] (-0.5537262719083826,0.9358277614604361)-- (-0.3004268259698176,0.863281154818844);
\draw [line width=2pt] (0.5537262719083828,0.935827761460436)-- (0.30042682596981773,0.863281154818844);
\draw [shift={(0.2012686136287686,0)},line width=2pt]  plot[domain=1.7720282133426348:2.0972398622433923,variable=\t]({1*0.9993903933137115*cos(\t r)+0*0.9993903933137115*sin(\t r)},{0*0.9993903933137115*cos(\t r)+1*0.9993903933137115*sin(\t r)});
\draw [shift={(0.2012686136287686,0)},line width=0.8pt]  plot[domain=-4.185945444936194:1.0067138651545406,variable=\t,samples=100]({1*0.9984751706373953*cos(\t r)+0*0.9984751706373953*sin(\t r)},{0*0.9984751706373953*cos(\t r)+1*0.9984751706373953*sin(\t r)});
\draw [shift={(0.2012686136287686,0)},line width=2pt]  plot[domain=1.0067138651545406:1.2106003310748068,variable=\t]({1*1.0016155636933795*cos(\t r)+0*1.0016155636933795*sin(\t r)},{0*1.0016155636933795*cos(\t r)+1*1.0016155636933795*sin(\t r)});
\draw [shift={(0.2012686136287686,0)},line width=2pt]  plot[domain=1.2106003310748068:1.7720282133426348,variable=\t]({1*1*cos(\t r)+0*1*sin(\t r)},{0*1*cos(\t r)+1*1*sin(\t r)});
\draw [shift={(-0.2012686136287686,0)},line width=2pt]  plot[domain=0.7341136510527686:1.0443527913464006,variable=\t]({1*1*cos(\t r)+0*1*sin(\t r)},{0*1*cos(\t r)+1*1*sin(\t r)});
\draw [shift={(-0.2012686136287686,0)},line width=2pt]  plot[domain=1.0443527913464006:1.3665976373458306,variable=\t]({1*0.9984751706373953*cos(\t r)+0*0.9984751706373953*sin(\t r)},{0*0.9984751706373953*cos(\t r)+1*0.9984751706373953*sin(\t r)});
\draw [shift={(-0.2012686136287686,0)},line width=2pt]  plot[domain=1.3665976373458306:1.9309923225149863,variable=\t]({1*1*cos(\t r)+0*1*sin(\t r)},{0*1*cos(\t r)+1*1*sin(\t r)});
\draw [shift={(-0.2012686136287686,0)},line width=0.8pt]  plot[domain=-4.3521929846646:0.7341136510527686,variable=\t,samples=100]({1*1*cos(\t r)+0*1*sin(\t r)},{0*1*cos(\t r)+1*1*sin(\t r)});
\begin{scriptsize}
\draw [fill=black] (-0.2012686136287686,0) circle (1pt);
\draw [fill=black] (0.001513949107973231,0.9792237906882774) circle (1pt);
\draw [fill=black] (0.5411562225186984,0.6699293714052301) circle (1pt);
\draw [fill=black] (0.7367730079530657,0.8464447891576709) circle (1pt);
\draw [fill=black] (0.2012686136287686,0) circle (1pt);
\draw [fill=black] (-0.3004268259698176,0.863281154818844) circle (1pt);
\draw [fill=black] (-0.5537262719083826,0.9358277614604361) circle (1pt);
\draw [fill=black] (0.5537262719083828,0.935827761460436) circle (1pt);
\draw [fill=black] (0.30042682596981773,0.863281154818844) circle (1pt);
\node[above] at (f) {$f$};
\node[below left] at (c2) {$c_2$};
\node[below right] at (c12) {$c_{1,2}$};
\node[above left] at (v1) {$v_1$};
\node[below right] at (q1) {$q_1$};
\node[above] at (v1') {$v_1'$};
\node[below left] at (q1') {$q_1'$};
\node[right] at (m2) {$m_2$};
\node[above right] at (t2) {$t_2$};
\end{scriptsize}
\end{tikzpicture}
\caption{}
\label{fig:reflection}
\end{figure}

\end{proof}

Next, we want to prove the monotonicity of the area of $T_{w,r}$ in $r$. %For that, we need the following area formula for triangles. For the proof of the spherical version, visit the PhD thesis of Sagmeister \cite{phd}, while the proof of the hyperbolic formula can be found in the PhD thesis of Frenkel \cite{Fre18}.

%\begin{lemma}\label{lemma:triangle-area}
%Let $T\subset\M^2$ a triangle of side lengths $a,b,c$ and opposite angles $\alpha,\beta,\gamma$. Then for the area of $T$ we have the following equations depending on $\M^2$:
%$$\cot\left(\frac{\area\left(T\right)}{2}\right)=\frac{\cot\frac{a}{2}\cot\frac{b}{2}+\cos\gamma}{\sin\gamma}\text{ if }\M^2=\sph^2,$$
%$$\cot\left(\frac{\area\left(T\right)}{2}\right)=\frac{\coth\frac{a}{2}\coth\frac{b}{2}-\cos\gamma}{\sin\gamma}\text{ if }\M^2=\hyp^2.$$
%\end{lemma}

\begin{lemma}\label{lem:areaoftrianglesdecreasesinr}
Let $0<w\leq r$, and also assume $r<\frac{\pi}{2}$ if $\M^2=\sph^2$. Then $\area(T_{w,r})$ decreases in $r$.
\end{lemma}

\begin{proof}

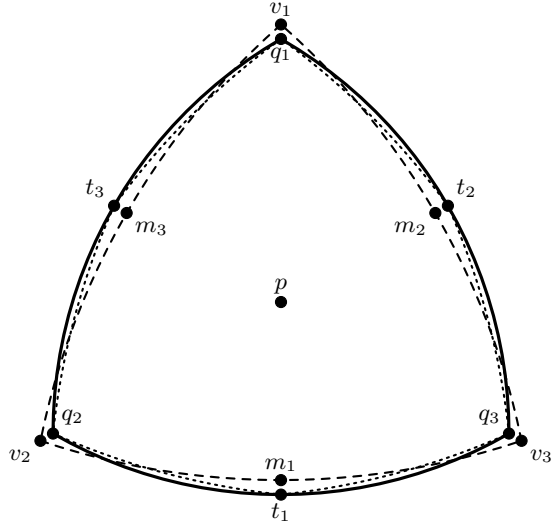
\begin{figure}[h]
\centering
\begin{tikzpicture}[line cap=round,line join=round,>=triangle 45,x=1cm,y=1cm,scale=3.5]
\clip(-1.1,-0.9) rectangle (1.1,1.2);
\coordinate (p) at (0,0);
\coordinate (v1) at (0,1.0550685001861666);
\coordinate (v2) at (-0.9137161238939671,-0.5275342500930831);
\coordinate (v3) at (0.9137161238939671,-0.5275342500930831);
\coordinate (q1) at (0,1);
\coordinate (q2) at (-0.8660254037844387,-0.5);
\coordinate (q3) at (0.8660254037844387,-0.5);
\coordinate (m1) at (0,-0.6769823073827106);
\coordinate (m2) at (0.5862838761060329,0.33849115369135513);
\coordinate (m3) at (-0.5862838761060329,0.33849115369135513);
\coordinate (t1) at (0,-0.7320508075688772);
\coordinate (t2) at (0.6339745962155613,0.36602540378443843);
\coordinate (t3) at (-0.6339745962155613,0.36602540378443843);
\draw [shift={(0,1)},line width=1.2pt]  plot[domain=4.1887902047863905:5.235987755982988,variable=\t]({1*1.7320508075688772*cos(\t r)+0*1.7320508075688772*sin(\t r)},{0*1.7320508075688772*cos(\t r)+1*1.7320508075688772*sin(\t r)});
\draw [shift={(-0.8660254037844387,-0.5)},line width=1.2pt]  plot[domain=0:1.0471975511965976,variable=\t]({1*1.7320508075688772*cos(\t r)+0*1.7320508075688772*sin(\t r)},{0*1.7320508075688772*cos(\t r)+1*1.7320508075688772*sin(\t r)});
\draw [shift={(0.8660254037844384,-0.5)},line width=1.2pt]  plot[domain=2.0943951023931953:3.141592653589793,variable=\t]({1*1.7320508075688776*cos(\t r)+0*1.7320508075688776*sin(\t r)},{0*1.7320508075688776*cos(\t r)+1*1.7320508075688776*sin(\t r)});
\draw [dashed,shift={(-1.8974127181395817,-1.095471743581706)},line width=0.8pt]  plot[domain=0.1993484732798451:0.8478490779167521,variable=\t]({1*2.867925794546124*cos(\t r)+0*2.867925794546124*sin(\t r)},{0*2.867925794546124*cos(\t r)+1*2.867925794546124*sin(\t r)});
\draw [dashed,shift={(1.8974127181395855,-1.0954717435817094)},line width=0.8pt]  plot[domain=2.2937435756730413:2.942244180309947,variable=\t]({1*2.8679257945461294*cos(\t r)+0*2.8679257945461294*sin(\t r)},{0*2.8679257945461294*cos(\t r)+1*2.8679257945461294*sin(\t r)});
\draw [dashed,shift={(0,2.190943487163417)},line width=0.8pt]  plot[domain=4.388138678066237:5.036639282703142,variable=\t]({1*2.8679257945461276*cos(\t r)+0*2.8679257945461276*sin(\t r)},{0*2.8679257945461276*cos(\t r)+1*2.8679257945461276*sin(\t r)});
\draw [dotted,shift={(0.3001370256235798,2.120126628859801)},line width=0.8pt]  plot[domain=4.29363517904471:4.60754400612637,variable=\t]({1*2.8679257945461276*cos(\t r)+0*2.8679257945461276*sin(\t r)},{0*2.8679257945461276*cos(\t r)+1*2.8679257945461276*sin(\t r)});
\draw [dotted,shift={(-0.3001370256235795,2.120126628859801)},line width=0.8pt]  plot[domain=4.817233954643008:5.1311427817246695,variable=\t]({1*2.8679257945461276*cos(\t r)+0*2.8679257945461276*sin(\t r)},{0*2.8679257945461276*cos(\t r)+1*2.8679257945461276*sin(\t r)});
\draw [dotted,shift={(1.986152032644239,-0.8001370256235804)},line width=0.8pt]  plot[domain=2.722838852249813:3.036747679331474,variable=\t]({1*2.8679257945461267*cos(\t r)+0*2.8679257945461267*sin(\t r)},{0*2.8679257945461267*cos(\t r)+1*2.8679257945461267*sin(\t r)});
\draw [dotted,shift={(1.6860150070206592,-1.3199896032362222)},line width=0.8pt]  plot[domain=2.199240076651514:2.5131489037331747,variable=\t]({1*2.8679257945461276*cos(\t r)+0*2.8679257945461276*sin(\t r)},{0*2.8679257945461276*cos(\t r)+1*2.8679257945461276*sin(\t r)});
\draw [dotted,shift={(-1.6860150070206603,-1.319989603236221)},line width=0.8pt]  plot[domain=0.6284437498566177:0.9423525769382786,variable=\t]({1*2.867925794546127*cos(\t r)+0*2.867925794546127*sin(\t r)},{0*2.867925794546127*cos(\t r)+1*2.867925794546127*sin(\t r)});
\draw [dotted,shift={(-1.98615203264424,-0.8001370256235789)},line width=0.8pt]  plot[domain=0.1048449742583187:0.41670518836069265,variable=\t]({1*2.867925794546127*cos(\t r)+0*2.867925794546127*sin(\t r)},{0*2.867925794546127*cos(\t r)+1*2.867925794546127*sin(\t r)});
\begin{scriptsize}
\draw [fill=black] (0,0) circle (0.6pt);
\draw [fill=black] (0,1) circle (0.6pt);
\draw [fill=black] (-0.8660254037844387,-0.5) circle (0.6pt);
\draw [fill=black] (0.8660254037844384,-0.5) circle (0.6pt);
\draw [fill=black] (0,1.0550685001861666) circle (0.6pt);
\draw [fill=black] (0,-0.7320508075688772) circle (0.6pt);
\draw [fill=black] (0,-0.6769823073827106) circle (0.6pt);
\draw [fill=black] (-0.9137161238939671,-0.5275342500930831) circle (0.6pt);
\draw [fill=black] (0.9137161238939666,-0.5275342500930836) circle (0.6pt);
\draw [fill=black] (0.5862838761060329,0.33849115369135513) circle (0.6pt);
\draw [fill=black] (-0.5862838761060327,0.3384911536913556) circle (0.6pt);
\draw [fill=black] (0.6339745962155613,0.36602540378443843) circle (0.6pt);
\draw [fill=black] (-0.633974596215561,0.3660254037844389) circle (0.6pt);
\node[above] at (p) {$p$};
\node[above] at (v1) {$v_1$};
\node[below left] at (v2) {$v_2$};
\node[below right] at (v3) {$v_3$};
\node[above] at (m1) {$m_1$};
\node[below left] at (m2) {$m_2$};
\node[below right] at (m3) {$m_3$};
\node[below] at (q1) {$q_1$};
\node[above right] at (q2) {$q_2$};
\node[above left] at (q3) {$q_3$};
\node[below] at (t1) {$t_1$};
\node[above right] at (t2) {$t_2$};
\node[above left] at (t3) {$t_3$};
\end{scriptsize}
\end{tikzpicture}
\caption{The regular $r$-disk triangle $T_{w,r}$ with inradius $\varrho$ (with thick boundary), the regular $\overline{r}$-disk triangle $T_{w,\overline{r}}$ with $r<\overline{r}$ (with dashed boundary) and the $\overline{r}$-disk hexagon $Q_{w,\overline{r},\varrho}$ (with dotted boundary) inscribed in $T_{w,r}$}
\label{fig:trianglesareadecreases}
\end{figure}

Let $\overline{r}>r$ (also $\overline{r}<\frac{\pi}{2}$ if $\M^2=\sph^2$). Let $\varrho$ be the inradius of $T_{w,r}$ and $\overline{\varrho}$ be the inradius of $T_{w,\overline{r}}$. The vertices of $T_{w,r}$ are $v_1,v_2,v_3$ and the vertices of $T_{w,\overline{r}}$ are $q_1,q_2,q_3$. The midpoint of the arc of $T_{w,r}$ opposite to $v_1,v_2,v_3$ are $m_1,m_2,m_3$, respectively and the midpoint of the arc of $T_{w,\overline{r}}$ opposite to $q_1,q_2,q_3$ are $t_1,t_2,t_3$, respectively. Let us position $T_{w,r}$ and $T_{w,\overline{r}}$ such that $p$ is the incenter of both $T_{w,r}$ and $T_{w,\overline{r}}$ and $q_i$ is on the half-line emanating from $p$ containing $v_i$ for all $i\in\{1,2,3\}$; see Figure~\ref{fig:trianglesareadecreases}. From Corollary~\ref{cor:rho0:monotone}, we know that $\varrho>\overline{\varrho}$. We observe that $d(t_i,m_i)=\varrho-\overline{\varrho}$ and $d(v_i,q_i)=(w-\varrho)-(w-\overline{\varrho})=\varrho-\overline{\varrho}$ for all $i\in\{1,2,3\}$. Then, since $r<\overline{r}$, the $\overline{r}$-disk hexagon $Q_{w,\overline{r},\varrho}=[q_1,t_3,q_2,t_1,q_3,t_2]_r$ is properly contained in $T_{w,r}=[q_1,t_3,q_2,t_1,q_3,t_2]_{\overline{r}}$, but has a greater area than $T_{w,\overline{r}}$ by Lemma~\ref{lem:areaofhexagonsincreases}. Hence, $\area(T_{w,r})>\area(T_{w,\overline{r}})$.

%---

%Let $T_{w,r}$ and $T_{w,r_0}$ be regular $r$-ball triangles with radius $0<w \leq r_0 < r$.  We can construct both these triangles centered at a point $p$ in a similar way as described in lemma \ref{lem:areaofhexagonsincreases} such that $T_{w,r}=Q_{w,r,\varrho}$ and $T_{w,r_0}=Q_{w,r_0,\varrho_0}$, where $\varrho$ is the inradius of $T_{w,r}$ and $\varrho_0$ is the inradius of $T_{w,r_0}$.  Let $q_1, q_2, q_3, t_1, t_2,$ and $t_3$ be the corresponding points of $Q_{w,r_0,\varrho_0}$ we have that $d(p,q_i)=w-\varrho_0$ and $d(p,t_i)=\varrho_0$.  Consider the set $Q_{w,r,\varrho_0}$.  Clearly, we have points $q_1'$, $q_2'$, and $q_3'$ such that $d(p, q_i')=w-\varrho_0$ and $\angle\left(q_i,p,q_j\right)=\frac{2\pi}{3}$ so $q_i' = q_i$ for all $i=\{1,2,3\}$, and similarly, since $d(t_i,p)=\varrho_0=d(t_i',p)$ and $t_i'$ are collinear with $q_i$ and $p$, $t_i'=t_i$ for all $i=\{1,2,3\}$ as well.  From the symmetry of these objects we can limit our explorations to that of a single arc, without loss of generality, we examine the arcs from $q_1$ to $t_2$.  Since $r_0 < r$ the area bounded by $[p,t_2]$, $[p,q_1]$ and the $r_0$-arc from $q_1$ to $t_2$ contains the area bounded by $[p,t_2]$, $[p,q_1]$ and the $r$-arc from $q_1$ to $t_2$. So, the area of $T_{w,r_0}$ is greater than that of $Q_{w,r,\varrho_0}$, from the lemma \ref{lem:areaofhexagonsincreases} and the fact that $\varrho < \varrho_0$, the area of $T_{w,r}$ is less than that of $Q_{w,r,\varrho_0}$, thus the area of $T_{w,r}$ decreases in $r$.
\end{proof}

Now we are ready to prove Theorem \ref{thm:isominwidth:spindleconvex}.

\begin{proof}[Proof of Theorem \ref{thm:isominwidth:spindleconvex}]
Let $K\subset\M^2$ be an $r$-ball convex body of minimal width $w$ with $0<w\leq r$, and if $\M^2=\sph^2$ then we also assume $r<\frac{\pi}{2}$. Let $\varrho=\varrho(K)$. Then by Theorem \ref{thm:Blaschke}, $\varrho_0\leq\varrho\leq\frac{w}{2}$ where $\varrho_0$ is the inradius of $T_{w,r}$. If $\varrho=\frac{w}{2}$, then by containment, $\area(K)\geq\area(B)$ where $B$ is the incircle of $K$. We know from the isodiametric inequality and the Blaschke--Lebesgue inequality that $\area(B)>\area(T_{w,w})$. Finally, we have $\area(T_{w,w})\geq\area(T_{w,r})$ from Lemma \ref{lem:areaoftrianglesdecreasesinr}.

If $\varrho<\frac{w}{2}$, then there is a cap-domain $C$ contained in $K$ where $C$ is the $r$-ball convex hull of the incircle $B=B(p,\varrho)$ and three points $q_1,q_2,q_3$ at distance $w-\varrho$ from $p$ such that the caps of $C$ are pairwise non-overlapping (cf. Lemma \ref{lem:nonoverlapping}). Since these caps are non-overlapping, $\area(C)=\area(\widetilde{C})$ where $\widetilde{C}$ is the $r$-ball convex hull of $B$ and the points $\widetilde{q}_1,\widetilde{q}_2,\widetilde{q}_3$ such that $d(p,\widetilde{q}_i)=w-\varrho$ and that $\angle(\widetilde{q}_j,p,\widetilde{q}_k)=\frac{2\pi}{3}$ for all $\{i,j,k\}=\{1,2,3\}$. Let $Q$ be the $r$-ball convex hull of $q_1,q_2,q_3,t_1,t_2,t_3$ where $t_i$ is the intersection point of the line through $p$ and $q_i$ and $\widetilde{C}$ different from $q_i$. But then
$$
\area(K)\geq\area(C)=\area(\widetilde{C})\geq\area(Q)\geq\area(T_{w,r})
$$
with equality if and only if $K$ and $T_{w,r}$ are congruent by Lemma \ref{lem:areaofhexagonsincreases}. This concludes the proof.
\end{proof}

\section{Acknowledgements}
The authors thank Professor K\'aroly Bezdek for helpful ideas and suggestions.

The research of Ferenc Fodor and \'Ad\'am Sagmeister was supported by NKFIH project no.~150151. Project no.~150151 has been implemented with the support provided by the Ministry of Culture and Innovation of Hungary from the National Research, Development and Innovation Fund, financed under the ADVANCED\_24 funding scheme.

This research was also supported by project TKP2021-NVA-09. Project no. TKP2021-NVA-09 has been implemented with the support provided by the Ministry of Innovation and Technology of Hungary from the National Research, Development and Innovation Fund, financed under the TKP2021-NVA funding scheme.

\bibliography{bibliography}
\bibliographystyle{abbrv}

\end{document}